\documentclass{amsart}
\usepackage{amssymb}
\usepackage{amscd}
\usepackage{amsfonts}
\usepackage{latexsym}
\usepackage[all]{xy}
\usepackage{verbatim}

\newtheorem{theorem}{Theorem}[section]
\newtheorem{lemma}[theorem]{Lemma}
\newtheorem{proposition}[theorem]{Proposition}
\newtheorem{corollary}[theorem]{Corollary} 
\theoremstyle{definition}  
\newtheorem{definition}[theorem]{Definition}
\newtheorem{example}[theorem]{Example}
\newtheorem{conjecture}[theorem]{Conjecture}  

\newtheorem{remark}[theorem]{Remark}

\newcommand{\id}{\text{id}}

\newcommand{\Fun}{\text{Fun}}

\renewcommand{\Vec}{\operatorname{\operatorname{\mathsf{Vec}}}}
\DeclareMathOperator{\Pic}{\operatorname{\mathsf{Pic}}}
\DeclareMathOperator{\BrPic}{\operatorname{\mathsf{BrPic}}}

\DeclareMathOperator{\Aut}{\operatorname{\mathsf{Aut}}}
\DeclareMathOperator{\Out}{\operatorname{\mathsf{Out}}}
\DeclareMathOperator{\Inn}{\operatorname{\mathsf{Inn}}}
\DeclareMathOperator{\St}{\operatorname{\mathsf{Stab}}}

\DeclareMathOperator{\Rep}{\operatorname{\mathsf{Rep}}}
\DeclareMathOperator{\Hom}{\operatorname{\mathsf{Hom}}}

\DeclareMathOperator{\Sym}{\operatorname{\mathsf{Sym}}}
\DeclareMathOperator{\Aa}{\operatorname{\mathfrak{A}}}

\newcommand{\op}{\text{op}}

\newcommand{\B}{\mathcal{B}}
\newcommand{\C}{\mathcal{C}}
\newcommand{\D}{\mathcal{D}}

\newcommand{\Z}{\mathcal{Z}}

\renewcommand{\L}{\mathcal{L}}
\newcommand{\M}{\mathcal{M}}
\newcommand{\A}{\mathcal{A}}
\newcommand{\N}{\mathcal{N}}

\newcommand{\be}{\mathbf{1}}

\renewcommand{\be}{\mathbf{1}}

\newcommand{\bt}{\boxtimes}
\newcommand{\ot}{\otimes}

\newcommand{\beq}{\begin{equation}}
\newcommand{\eeq}{\end{equation}}

\newcommand{\bpf}{\begin{proof}}
\newcommand{\epf}{\end{proof}}

\newcommand{\bth}{\begin{theorem}}
\renewcommand{\eth}{\end{theorem}}
\newcommand{\bpr}{\begin{proposition}}
\newcommand{\epr}{\end{proposition}}
\newcommand{\ble}{\begin{lemma}}
\newcommand{\ele}{\end{lemma}}
\newcommand{\bco}{\begin{corollary}}
\newcommand{\eco}{\end{corollary}}
\newcommand{\bde}{\begin{definition}}
\newcommand{\ede}{\end{definition}}
\newcommand{\bex}{\begin{example}}
\newcommand{\eex}{\end{example}}
\newcommand{\bre}{\begin{remark}}
\newcommand{\ere}{\end{remark}}
\newcommand{\bcj}{\begin{conjecture}}
\newcommand{\ecj}{\end{conjecture}}

\newcommand{\GrpV}{\mathbb{Z}/2\mathbb{Z}\times \mathbb{Z}/2\mathbb{Z}}
\newcommand{\GrpCtwo}{\mathbb{Z}/2\mathbb{Z}}
\newcommand{\GrpCthree}{\mathbb{Z}/3\mathbb{Z}}

\begin{document}
\title[Brauer-Picard groups of pointed fusion categories]{Categorical Lagrangian Grassmannians and Brauer-Picard groups of pointed fusion categories}
\author{Dmitri Nikshych}
\address{D.N.: Department of Mathematics and Statistics,
University of New Hampshire,  Durham, NH 03824, USA}
\email{nikshych@math.unh.edu}
\author{Brianna Riepel}
\address{B.R.: Department of Mathematics and Statistics,
University of New Hampshire,  Durham, NH 03824, USA}
\email{briepel@wildcats.unh.edu}

\begin{abstract}

We analyze the action of the Brauer-Picard group of a pointed fusion category
on the set of Lagrangian subcategories of its center. Using this action we compute
the Brauer-Picard groups of pointed fusion categories associated to several classical 
 finite groups. As an application, we construct new examples of weakly group-theoretical
fusion categories.
\end{abstract}

\date{\today}

\maketitle

\setcounter{tocdepth}{1}
\tableofcontents
\newpage

\section{Introduction}


Let $\A$ be a fusion  category.
The  {\em Brauer-Picard} group $\BrPic(\A)$  of $\A$ consists of equivalence classes of invertible $\A$-bimodule 
categories  (see \cite{ENO2}).   Brauer-Picard  groups play an important role in the theory of fusion  categories.
In particular, they are used in the classification of graded extensions of fusion categories \cite{ENO2}.
In the case when $\A$ is the category of representations of a Hopf algebra the group  $\BrPic(\A)$   is known as the
{\em strong Brauer group}, see \cite{COZ}. 

Computing Brauer-Picard groups for concrete examples  of fusion categories is an important task.  
A number of special results of this type was obtained  in the literature, see, e.g., \cite{CC, GS, Mo}. 
In this paper we develop techniques that allow to compute explicitly Brauer-Picard groups of pointed
(and, hence, group-theoretical) fusion categories.  We use the following characterization
of Brauer-Picard groups established in \cite{ENO2}.
For any fusion category $\A$ there is a canonical isomorphism:
\begin{equation}
\label{categorical Phi}
\Phi: \BrPic(\A) \to  \Aut^{br}(\Z(\A)),
\end{equation}
where $\Z(\A)$ is the {\em Drinfeld center} of $\A$ and $\Aut^{br}(\Z(\A))$ is the group of braided
autoequivalences of $\Z(\A)$. The latter group has a distinct geometric flavor (e.g., when $\A$ is
the representation category of a finite abelian group $A$, the group  $\Aut^{br}(\Z(\A))$
is  the split orthogonal group
$O(A\oplus\widehat{A})$).  This suggests the use of   ``categorical-geometric" methods
for computation of $\Aut^{br}(\Z(\A))$ (which is identified with $\BrPic(\A)$ via 
isomorphism \eqref{categorical Phi}).

In this paper we analyze
the action of  $\Aut^{br}(\Z(\A))$, where $\A=\Vec_G$ is the category of vector spaces graded by a finite group $G$,  
on the categorical Lagrangian 
Grassmannian $\mathbb{L}(G)$ associated to it.  By definition, the latter is the set of Lagrangian subcategories of $\Z(\A)$. 
The set $\mathbb{L}(G)$ was described in group-theoretical terms in \cite{NN}.
We  determine the point stabilizers for this action and explicitly compute the corresponding permutation groups
in a number of concrete examples.  Note that Mombelli in \cite{Mo} studied the group $\BrPic(\Vec_G)$
using methods different from ours. 

Module categories over a  braided fusion category  $\C$ can be regarded as $\C$-bimodule categories. 
In this case the group $\BrPic(\C)$ contains a subgroup $\Pic(\C)$, called the {\em Picard group} of $\C$,
consisting of invertible $\C$-module categories \cite{ENO2}. This group  is isomorphic
to the group of Morita equivalence classes of Azumaya algebras in $\C$ (the latter group was introduced in
\cite{OZ}). 

One defines a homomorphism
\begin{equation}
\label{categorical partial}
\partial: \Pic(\C) \to  \Aut^{br}(\C),
\end{equation}
in a way parallel to \eqref{categorical Phi}.  
It was shown in \cite{ENO2} that \eqref{categorical partial} is an isomorphism for every non-degenerate 
braided fusion category~$\C$. One has $\Pic(\Z(\A))\cong \BrPic(\A)$ for any fusion category $\A$.


The paper is organized as follows. 

In Section~\ref{cohomology section} we collect results about finite group cohomology
that will be used for computations. Section~\ref{section: Fusion categories and their Brauer-Picard groups} 
contains definitions and  basic facts about fusion categories and their Brauer-Picard groups.

In Section~\ref{section: parameterization of BrPic} we present a useful parameterization of the group 
$\BrPic(\Vec_G)$ previously obtained by Davydov in \cite{Da2}.  This parameterization allows one easily 
recognize involutions in $\BrPic(\Vec_G)$ (see Corollary~\ref{order 2}). 

In Section~\ref{section: Braided autoequivalences of centers}  we describe, following \cite{ENO2},  
the construction of isomorphism   \eqref{categorical Phi} between the
Brauer-Picard group of a fusion category and the group of braided autoequivalences of its center.  This allows us
to concentrate on the computation of the latter group.  For a braided fusion category $\C$ we find
the subgroup of $\Aut^{br}(\Z(\C))$ stabilizing the subcategory $\C \subset \Z(\C)$ 
(see Proposition~\ref{stabilizer of C} and Corollary~\ref{found StRepG}).

The action of $\Aut^{br}(\Z(\Vec_G))$ on the Lagrangian Grassmannian $\mathbb{L}(G)$
(which is, by definition, the set of Lagrangian subcategories of $\Z(\Vec_G)$)
is studied in Section~\ref{action on Grassmannian}.  In general, this action is not transitive. 
We show in Proposition~\ref{transitive action} that the orbit of this action containing the canonical 
subcategory  $\Rep(G) \subset \Z(\Vec_G)$ is precisely the set $\mathbb{L}_0(G)$ 
of subcategories of $ \Z(\Vec_G)$ braided equivalent to $\Rep(G)$. 

Sections~\ref{examples: symmetric} through \ref{examples: dihedral} 
illustrate our techniques. They contain explicit computations
of groups $\Aut^{br}(\Z(\Vec_G))$ for several classes of finite groups $G$.  The common feature
of these examples is that in each case it is possible to describe the set $\mathbb{L}_0(G)$ and the 
corresponding action of  $\Aut^{br}(\Z(\Vec_G))$.  Combining information about this action
with previously developed machinery we  determine groups $\Aut^{br}(\Z(\Vec_G))$. 
As a byproduct, we obtain new examples of  non-integral weakly group-theoretical fusion categories. 

\textbf{Acknowledgments.}
We are grateful to Juan Cuadra, Alexei Davydov, Pavel Etingof, Deepak Naidu, and  Victor Ostrik
for helpful discussions and valuable comments.
The work of the first  named author  was partially supported  by the NSA grant H98230-13-1-0236
and NSF grant DMS-0800545.

\section{Conventions and  notation}

Throughout this paper we work over an algebraically closed field  $k$ of characteristic $0$. 
All categories considered in this paper are finite, abelian, semisimple, and $k$-linear.
All functors are additive and $k$-linear.
We freely use the language and basic results of the theory of fusion categories
and module categories over them \cite{ENO1, ENO2, DGNO}.
We will denote $\Vec$ the fusion category of $k$-vector spaces. 

For a finite group $G$ we denote $\Aut(G)$ the group of automorphisms of $G$ and by $\Out(G)$
the group of (congruence classes of) outer automorphisms of $G$.  For a $G$-module $A$
we denote by $Z^n(G,\,A)$ the group of $n$-cocycles on $G$ with values in $A$ and by $H^n(G,\,A)$
the corresponding $n$th cohomology group. We will often identify cohomology classes
with cocycles representing them. 

For any (not necessarily Abelian) group $G$ we denote $\widehat{G} = \Hom(G,\, k^\times)$ 
the group of linear characters of $G$. 

We can view  the multiplicative group $k^\times$ as a $G$-module with the trivial action.
There is an obvious  action of $\Aut(G)$ on  $H^n(G,\,k^\times)$. This action factors through
the subgroup of inner automorphisms and, hence, gives rise to an action of $\Out(G)$.
For a  subgroup $L \subset G$, a cocycle   $f\in Z^n(L,\,k^\times)$, and an automorphism
$\theta\in \Aut(G)$ denote
\begin{equation}
 f^{\theta}=f \circ (\theta^{-1} \times \cdots \times \theta^{-1})  \in Z^n(\theta(L),\,k^\times).
\end{equation}
It is clear that the cohomology class of $f^\theta$ in $H^n(L,\,k^\times)$ is well defined.
When $\theta$ is the inner automorphism $x \mapsto gxg^{-1},\,x\in G$, we denote
$f^\theta$ by $f^g$.

For any positive integer $n$ we denote $D_{2n}\cong \mathbb{Z}/n\mathbb{Z} \rtimes \mathbb{Z}/2\mathbb{Z}$ 
the dihedral  group of order $2n$, $S_n$ the symmetric group of degree $n$, and $A_n$
the alternating group of degree $n$. More generally, for any set $\Omega$ we denote $\Sym(\Omega)$ the symmetric
group of~$\Omega$. 

Finally, for a finite group $G$ we denote by $\Vec_G$  the fusion category of finite-dimensional 
$G$-graded vector spaces and by $\Rep(G)$
the symmetric fusion category of finite-dimensional representations of $G$. 

\section{Some facts about cohomology of finite groups}
\label{cohomology section}

Let $G$ be a finite group.
\begin{remark}
\label{H1(G,A)}
Let $A$ be a $G$-module. It is well known  that $H^1(G,\,A)$ classifies
homomorphisms $G\to A\rtimes G$ which are right inverse to the standard projection $A\rtimes G\to G$,
up to a conjugation by elements of  $A$.
\end{remark}

\begin{proposition}
\label{cohomofprimeordergroups}
 Let $G$ be a finite group and let $A$ be a finite $G$-module  such that the orders $|G|$ and $|A|$ are relatively
 prime. Then $H^n(G,\,A)=0$ for all $n$.
\end{proposition}


The following result is taken from \cite{T} and \cite[Theorem 2.2.5]{K}. 
It can also be proved by means of the Hochschild-Serre spectral sequence.
\begin{theorem}
\label{5termexactseq}
Let $G= N \rtimes T$ and let  $\tilde{M}(G)\subset H^2(G,\,k^{\times})$ be the kernel of the restriction  homomorphism
$H^2(G,\,k^{\times})\rightarrow H^2(T,\,k^{\times})$. Then
\begin{equation}
\label{M(semidirect)}
H^2(G,\,k^{\times}) \cong H^2(T,\,k^{\times})\times \tilde{M}(G)
\end{equation}
and there is  an exact sequence
  \begin{equation}
  \label{sequence for semi direct}
  0\to  H^1(T,\, \widehat{N})\to \tilde{M}(G)\xrightarrow{res}  H^2(N,\,k^{\times})^{T}\to H^2(T,\, \widehat{N}),
  \end{equation}
where the homomorphism $\text{res}: \tilde{M}(G)\to H^2(N,\,k^{\times})^{T}$ is induced by the restriction 
$H^2(G,\,k^{\times})\to H^2(N,\,k^{\times})$.
\end{theorem}

The following result \cite[Theorem 2.1.2]{K} will be useful for our computations.

\begin{theorem}
\label{Sylow restriction}
Let $G$ be a finite group and let $P$ be a Sylow $p$-subgroup of $G$. The restriction map
$H^2(G,\,k^{\times})\to H^2(P,\,k^{\times})$ is injective on the $p$-primary component of $H^2(G,\,k^{\times})$. 
\end{theorem}

\section{Fusion categories and their Brauer-Picard groups} 
\label{section: Fusion categories and their Brauer-Picard groups}

For a fusion category $\A$ let $\Aut(\A)$ denote the group of isomorphism classes
of tensor autoequivalences of $\A$. It is known that this group is finite \cite{ENO1}.

A fusion category is called {\em pointed} if all its simple objects are invertible with respect
to the tensor product.  A most general example of a pointed fusion category is the category $\Vec_G^\omega$
of vector spaces graded by a finite group $G$ with the associativity constraint given by a $3$-cocycle
$\omega \in Z^3(G,\, k^\times)$.  In this paper we only consider the case when $\omega$
is cohomologically trivial, i.e., we work with pointed categories of the form $\Vec_G$. 
Let $\delta_g,\, g\in G,$ denote simple objects  of $\Vec_G$. We have  $\delta_g \ot \delta_h  \cong \delta_{gh}$.
In particular,  the unit object of $\Vec_G$  is $\delta_1$. 

The following result is well known.

\begin{proposition}
\label{Aut(VecG)}
Let $G$ be a finite group. Then $\Aut(\Vec_G) \cong  H^2(G,\, k^\times) \rtimes \Aut(G)$.
\end{proposition}

For  $\zeta \in H^2(G,\, k^\times)$ and $a \in \Aut(G)$ the corresponding
autoequivalence  $F_{(a,\zeta)}$  of $\Vec_G$ is defined as follows. As a functor,
$F_{(a,\zeta)}(\delta_g) = \delta_{a(g)}$, while the tensor structure of $F_{(a,\zeta)}$
is given by
\[
\zeta(g,\,h)\id_{\delta_{a(gh)}}: 
F_{(a,\zeta)}(\delta_g) \otimes F_{(a,\zeta)}(\delta_h) \xrightarrow{\sim} F_{(a,\zeta)}(\delta_{gh}), \qquad g,h\in G.
\] 

Let $\A$ be a fusion category. The notion of a tensor product $ \bt_\A$ 
of $\A$-bimodule categories was introduced in \cite{ENO2}.  With respect to this product
equivalence classes of $\A$-bimodule categories form a monoid. The unit of this monoid
is the regular $\A$-bimodule category $\A$.   An $\A$-bimodule category $\M$ is called
{\em invertible} if there is an $\A$-bimodule category $\N$ such that $\M \bt_\A \N \cong \A$
and $\N \bt_\A \M \cong \A$.  By definition, the {\em Brauer-Picard} group of $\A$
is the group $\BrPic(\A)$ of equivalence classes of invertible $\A$-bimodule categories. 

The Brauer-Picard group is an important invariant of a fusion category. It is used, in particular,
in the classification of extensions of fusion categories \cite{ENO2}.  Let $G$ be a finite group.
By a {\em $G$-extension} of a fusion category $\A$  we mean a faithfully $G$-graded fusion category
\begin{equation}
\label{Extension}
\B =\bigoplus_{g\in G}\, \B_g,\qquad \text{ with }  \B_e\cong \A.
\end{equation}
Such extensions are parameterized by group homomorphisms $c: G\to \BrPic(\A)$ and certain 
cohomological data associated to $G$ (provided that certain obstructions vanish, see \cite{ENO2} for details).  
One has $c(g)=\B_g$ for all $g\in G$. We say that an extension \eqref{Extension} is {\em non-trivial}
if $\B_g \not\cong \A$ (as a left $\A$-module category) for some $g\in G$. 

\begin{remark}
\label{rel prime ext}
In a particularly simple situation when $|G|$ and the Frobenius-Perron dimension of $\A$ are relatively
prime,  for any fixed homomorphism $c:G\to \BrPic(\A)$ extensions \eqref{Extension}  exist and 
are parameterized  by a torsor over $H^3(G,\, k^\times)$, see \cite[Theorem 9.5]{ENO2}. 
\end{remark}

\section{Parameterization of $\BrPic(\Vec_G)$}
\label{section: parameterization of BrPic}

Let $G$ be a finite group.
In this Section we recall a group-theoretical parameterization of the Brauer-Picard group of $\Vec_G$.
This description was obtained by Davydov in \cite{Da2} (in terms of equivalences of centers, cf.\ isomorphism 
\eqref{categorical Phi}). We provide an alternative argument for the reader's convenience.

Recall \cite{O} that  indecomposable  $\Vec_G$-module categories are parameterized by pairs  $(L,\, \mu)$, 
where $L \subset G$ is a subgroup and $\mu \in Z^2(L,\, k^\times)$.  Namely, the category $\M(L,\, \mu)$ 
corresponding to  such a pair  consists of vector spaces graded by the set of cosets $G/L$ with the action of $\Vec_G$ 
induced by the translation action of $G$ on $G/L$ and the module category structure induced  by $\mu$.

Two $\Vec_G$-module categories $\M(L,\, \mu)$  and $\M(L',\, \mu')$ are equivalent if and only if
there is $g\in G$ such that $L'=gLg^{-1}$ and $2$-cocycles $\mu'$ and $\mu^g$ are cohomologous
in $H^2(L',\, k^\times)$. 

Fix a subgroup $L$ of $G$.
Let $E$ denote the group of isomorphism classes  of right $\Vec_G$-module 
autoequivalences of   $\M(L,\, \mu)$  isomorphic to the identity as an additive functor.
It follows from \cite{N} that there is a group isomorphism  
\begin{equation}
\label{E to Hom}
\iota: E \to \widehat{L} : F \mapsto \iota_F
\end{equation}
such that the $\Vec_G$-module functor structure $\delta_x\ot F(L) \xrightarrow{\sim} F(\delta_x \ot L) = F(xL)$
is given by $\iota_F(x) \id_{F(xL)}$ for all $x\in L$.

Let $G_1,\, G_2$ be a pair of normal subgroups of $G$ centralizing each other. Let us define
\[
L_1:= G_1 \cap L,\, L_2 :=G_2 \cap L.
\]
Any  $2$-cocycle $\mu\in Z^2(L,\,k^\times)$ determines a group homomorphism  
\begin{equation}
\label{a alt}
a:  L_1 \to \widehat{L_2} :  g \mapsto a_g, 
\mbox{ where } a_g(h) :=\frac{\mu(g, h)}{\mu(h, g)},\, h\in L_2.
\end{equation}
Similarly, $\mu$ determines a group homomorphsim $L_2 \to \widehat{L_1}$.

\begin{lemma}
\label{when id}
Let $G_1,\, G_2$ be a pair of commuting normal subgroups of $G$ such that
$G_1L = G_2 L= G$. For $g\in G_1$ let $F_g$ denote the functor of left tensor 
multiplication by $\delta_g$ on $\M(L,\, \mu)$. Then $F_g$ is equivalent to
the identity as a left $\Vec_{G_2}$-module
autoequivalenvce of  $\M(L,\, \mu)$ if and only if $g\in L_1$ and
$a_g =1$ on $L_2$.
\end{lemma}
\begin{proof}
It is clear that $F_g$ is isomorphic to $\id_{\M(L,\, \mu)}$ as an additive
functor if and only if $g\in  L_1$.
Let $C =\Hom_L(G,\, k^\times)$ and let
$\tilde{\mu}\in Z^2(G,\, C)$ be a $2$-cocycle such that the cohomology
class of $\tilde{\mu}$ in $H^2(G,\, C)$ is identified with the class of $\mu$
in $H^2(L,\, k^\times)$ via Shapiro's Lemma, i.e.,
$\mu(h_1,\, h_2) = \tilde\mu(h_1,\, h_2)(L)$ for all $x_1,x_2\in L$. 

Then the $\Vec_{G_2}$-module
structure on $F_g$ is given by 
\[
\tfrac{\tilde\mu(g,\, g_2)(xL)}{\tilde\mu(g_2,\, g)(xL)}\, \id_{g_2xL}: 
F_g( \delta_{g_2} \ot xL) \xrightarrow{\sim} \delta_{g_2} \ot F_g(xL),\qquad x,\,g_2\in L_2. 
\]
Using isomorphism  \eqref{E to Hom} we conclude that for $g\in L_1$ 
one has $F_g \cong \id$ as
a left $\Vec_{G_2}$-module functor if and only if $a_g=1$, as required.
\end{proof}

Let $\mu$ be a $2$-cocycle on $G$.  Let $L_1,\, L_2 \subset G$
be a pair of subgroups centralizing each other.  It is straightforward to check
that the function
\begin{equation}
\label{def alt}
Alt(\mu) : L_1\times L_2 \to k^\times : (x_1,\, x_2) \mapsto \frac{\mu(x_1,\, x_2)}{\mu(x_2,\, x_1)},
\end{equation}
is a bicharacter, i.e., is multiplicative in both arguments.

Note that a $\Vec_G$-bimodule category is the same thing as a $\Vec_{G \times G^\op}$-module category,
where $G^\op$ is $G$ with the opposite multiplication. 

\begin{proposition}
\label{M(L,mu) inv}
Let $G$ be a finite group, let   $L$ be a subgroup of  $G \times G^\op$,
and let $\mu \in Z^2(L,\, k^\times)$ be a $2$-cocycle. 
Then $\Vec_G$-bimodule category  $\M(L,\,\mu)$ is invertible if and only if
the following  three conditions are satisfied:
\begin{enumerate}
\item[(i)] $L (G \times \{1\}) = L(\{ 1\} \times G^\op) = G \times G^\op$, 
\item[(ii)] $L_1:= L \cap (G \times \{1\}) $ and $L_2:= L \cap (\{ 1\} \times G^\op)$ are Abelian  groups, 
\item[(iii)] bicharacter  $Alt(\mu) : L_1 \times L_2 \to k^\times$ defined in \eqref{def alt}
is non-degenerate.
\end{enumerate}
\end{proposition}
\begin{proof}
Let us denote $\M:= \M(L,\,\mu)$.
Condition (i) is equivalent to $(G \times G^\op)/L$ being transitive
as both left and right $G$-set, i.e., to $\M$ being indecomposable
as left and right $\Vec_G$-module category. 
It implies that $L_1$ is a normal subgroup of $G$ and 
$L_2$ is a normal subgroup of $G^\op$. 

For $g\in G$ let $L(g)$ (respectively, $R(g)$) denote the additive endofunctor
of $\M$ given by the action of $\delta_g \bt 1 $ (respectively, $1\bt \delta_g$).
By \cite{ENO2}  $\M$  is invertible if and only if
the functors $\Vec_G \to  \Fun_{\Vec_G}(\M,\, \M):  g \mapsto R(g)$ 
(respectively, $\Vec_G \to  \Fun(\M,\, \M)_{\Vec_G}:  g \mapsto L(g)$)
are equivalences.  Since those functors are tensor, the above conditions
are equivalent to $L(g) \not\cong \id_\M$ as a right $\Vec_G$-module functor 
(respectively,  to $R(g) \not\cong \id_\M$ as a left $\Vec_G$-module functor)
for all $g\neq 1$.

We apply Lemma~\ref{when id} with $G$ replaced by $G\times G^\op$,
$G_1 = G \times \{1\}$ and $G_2= \{ 1\} \times G^\op$.   It follows that
the above conditions are satisfied if and only if group homomorphisms
defined as in \eqref{a alt}, i.e., 
\begin{equation}
\label{rho1}
 L_1 \to \widehat{L_2} :  x \mapsto a_x, 
\mbox{ where } a_x(h) :=\frac{\mu(x, h)}{\mu(h, x)},\, h\in L_2,
\end{equation}
and 
\begin{equation}
\label{rho2}
 L_2 \to \widehat{L_1} :  y \mapsto a'_y, 
\mbox{ where } a'_y(g) :=\frac{\mu(y, g)}{\mu(g, y)},\, g\in L_1,
\end{equation}
are injective. This is equivalent to $L_1,\, L_2$ being Abelian and  
$Alt(\mu)$ being non-degenerate on $L_1\times L_2$.
\end{proof}

\begin{remark}
\label{muresinv}
For an invertible $\Vec_G$-bimodule category  $\M(L,\,\mu)$  the subgroups $L_1\subset G$
and $L_2 \subset G^\op$ are normal  and restrictions $\mu|_{L_1\times L_1}$
and $\mu|_{L_2\times L_2}$ are $G$-invariant.
\end{remark}

\begin{remark}
\label{restriction to one-sided}
It is easy to  describe a one-sided restriction of the  $\Vec_G$-module category $\M(L,\,\mu)$
from Proposition~\ref{M(L,mu) inv}.  Namely, as a left $\Vec_G$-module category
it is equivalent to $\M(L_1,\, \mu|_{L_1\times L_1})$.   
\end{remark}


There is a convenient way to determine which of the categories $\M(L,\, \mu)$ 
 described in Propositon~\ref{M(L,mu) inv} are involutions in the Brauer-Picard group.  


\begin{remark}
\label{Mop}
Let $G$ be a finite group, let   $L$ be a subgroup of  $G \times G^\op$,
and let $\mu \in Z^2(L,\, k^\times)$ be a $2$-cocycle satisfying conditions
of Proposition~\ref{M(L,mu) inv}. Then the inverse of $\M(L,\, \mu)$ in $\BrPic(\Vec_G)$
is  $\M(L^\vee, \, (\mu^\vee)^{-1})$, where
\begin{eqnarray*}
L^\vee &=& \{ (x_2,\, x_1) \mid (x_1,\, x_2) \in L \}, \\
\mu^\vee(  (x_1,\, x_2) ,\, (y_1,\, y_2)   ) &=& \mu ((x_2^{-1},\, x_1^{-1}),\, (y_2^{-1},\, y_1^{-1})).
\end{eqnarray*} 
Indeed, it was shown in \cite{ENO2} that  the inverse of a bimodule category is given by taking
its opposite. 
\end{remark}

\begin{corollary}
\label{order 2}
The category $\M(L,\, \mu)$  has order $\leq 2$ in $\BrPic(\Vec_G)$ if there is $g\in G \times G^\op$
such that $L^\vee = gLg^{-1}$ and $\mu^g$ and $(\mu^\vee)^{-1}$ are cohomologous in $H^2(L^\vee,\,k^\times)$.
\end{corollary}

\section{Braided autoequivalences of centers}
\label{section: Braided autoequivalences of centers}

Let $\C$ be a braided fusion category with braiding $c_{X,Y}: X \ot Y \xrightarrow{\sim} Y \ot X$.  
Let $\D$ be a fusion subcategory of $\C$.  The {\em centralizer} of $\D$ in $\C$  \cite{Mu1} is the fusion subcategory
$\D'\subset \C$ consisting of objects $X$ such that $c_{YX}\circ c_{XY} =\id_{X\ot Y}$ for all objects $Y$ in $\C$. 
A braided fusion category $\C$ is {\em symmetric} if $\C=\C'$ and {\em non-degenerate} if $\C' =\Vec$.
A symmetric fusion category is called {\em Tannakian} if it is equivalent to $\Rep(G)$, the category of representations
of a finite group $G$. 

Let $\Aut^{br}(\C)$ denote the group of isomorphism classes of braided autoequivalences of $\C$.
The following result is well known.

\begin{proposition}
\label{Aut-br=Out}
Let $G$ be a finite group.  We have $\Aut^{br}(\Rep(G)) \cong \Out(G)$.
\end{proposition}
\begin{proof}
By the result of Deligne \cite{De},  every braided tensor functor $F: \Rep(G) \to \Vec$ is isomorphic
to the obvious forgetful functor. Furthermore,  the group $G_F$ of tensor automorphisms of $F$ is
isomorphic to $G$.  Hence, a braided tensor autoequivalence  $\alpha \in \Aut^{br}(\Rep(G))$ induces
a group automorphism $\iota(\alpha)\in \Aut(G_F)$. The assignment $V \mapsto F(V)$ is a braided
tensor equivalence between $\Rep(G)$ and $\Rep(G_F)$. Under this equivalence $\alpha$
corresponds to the autoequivalence of $\Rep(G_F)$ induced by the automorphism $\iota(\alpha)$.

Hence, every braided autoequivalence of $\Rep(G)$ is induced by an automorphism of $G$. 
It is straightforward to verify that the above autoequivalence $\iota(\alpha)$ is isomorphic to the identity
tensor functor if and only if the corresponding group automorphism is inner.  This implies the result.
\end{proof}

For any fusion category $\A$ let $\Z(\A)$ denote its {\em center}.  The objects of $\Z(\A)$ are pairs
$(Z,\,\gamma)$ where $Z$ is an object of $\A$  and $\gamma=\{\gamma_X\}_{X\in \A}$, where
\[
\gamma_X: X\ot Z \xrightarrow{\sim} Z\ot X,
\]
is a natural isomorphism  
satisfying certain compatibility conditions.  We will usually simply write $Z$ for $(Z,\,\gamma)$. 
It is known that $\Z(\A)$ is a non-degenerate braided fusion category \cite{Mu2, DGNO}. 

Let $\A$ be a fusion category and let $\M$ be an invertible $\A$-bimodule category.
One assigns to $\M$ a braided autoequivalence $\Phi_\M$ of $\Z(\A)$ as follows.
Note that  $\Z(\A)$ can be identified with the category of $\A$-bimodule endofunctors of $\M$ in two ways:
via the functors $Z \mapsto Z\ot -$ and $Z \mapsto - \ot Z$.  Define $\Phi_\M$ in such a way that
there is an isomorphism of $\A$-bimodule functors 
\begin{equation}
\label{PhiM}
Z\ot - \cong -\ot \Phi_\M(Z) 
\end{equation}
for all $Z\in \Z(\A)$.

The following result was established in \cite{ENO2}.

\begin{theorem}
\label{BrPic=Aut-br}
Let $\A$ be a fusion category. The assignment 
\begin{equation}
\label{M to phiM}
\M \mapsto \Phi_\M
\end{equation}
 gives rise to
an isomorphism 
\begin{equation}
\label{main iso}
\BrPic(\A) \simeq \Aut^{br}(\Z(\A)). 
\end{equation}
\end{theorem}

\begin{remark}
\label{discussing inverse}
The inverse to the above isomorphism \eqref{M to phiM} is constructed as follows
(see \cite[Section 5]{ENO2}). Let $I: \A \to \Z(\A)$ be the right adjoint of the forgetful functor
$F: \Z(\A)\to \A$.  For a braided autoequivalence $\alpha\in \Aut^{br}(\Z(\A))$  consider
the commutative algebra $A:= \alpha^{-1}(I(\be))$ in $\Z(\A)$.  Let $\M_\alpha$ be any indecomposable
component of the category of $F(A)$-modules in $\A$. It has a structure of invertible $\A$-bimodule category 
and the assignment $\alpha \mapsto \M_\alpha$ is the inverse of \eqref{M to phiM}.
\end{remark}

Let $A$ be a finite Abelian group. Then Theorem~\ref{BrPic=Aut-br} implies that
\begin{equation}
\BrPic(\Vec_A)  \cong O(A \oplus \widehat{A},\, q), 
\end{equation}
where $O(A \oplus \widehat{A},\,q)$
is the group of automorphisms of $A  \oplus \widehat{A}$ preserving the
canonical quadratic form
\[
q(a,\,\chi) =\chi(a),\qquad a\in A,\, \chi\in \widehat{A}.
\]

For any fusion category $\A$ there is an   induction homomorphism
\begin{equation}
\label{induction from Aut}
\Delta: \Aut(\A) \to \Aut^{br}(\Z(\A)) : \alpha \mapsto \Delta_\alpha,
\end{equation}
where $\Delta_\alpha(Z,\, \gamma) = (\alpha(Z),\, \gamma^\alpha)$ and   $\gamma^\alpha$
is defined by the following commutative diagram
\begin{equation} 
 \xymatrix{
 X\otimes \alpha(Z)\ar[rr]^{\gamma^\alpha_X}\ar[d] & &  \alpha(Z)\otimes X\ar[d] \\
 \alpha(\alpha^{-1}(X)) \otimes \alpha(Z)\ar[d]_{J_{\alpha^{-1}(X),Z}} & & \alpha(Z) \otimes \alpha(\alpha^{-1}(X))\ar[d]^{J_{Z,\alpha^{-1}(X)}}\\
 \alpha(\alpha^{-1}(X)\otimes Z)\ar[rr]^{\alpha(\gamma_{\alpha^{-1}(X)})} & & \alpha(Z\otimes \alpha^{-1}(X)).
 }
\end{equation}
Here $\alpha^{-1}$ is a quasi-inverse of $\alpha$ and $J_{X,Y}: \alpha(X)\ot \alpha(Z) \xrightarrow{\sim} \alpha(X\ot Z)$
is the tensor functor structure of $\alpha$. 

\begin{example}
\label{induced bimodule}
Let $\A$ be a fusion category and let $\alpha\in \Aut(\A)$. Consider an invertible $\A$-bimodule category $\A_\alpha$,
where $\A_\alpha =\A$ and the actions of $\A$ on $\A_\alpha$ are  given by
\begin{equation}
\label{C alpha}
(X,\, V) \mapsto \alpha(X) \ot V,\qquad  (V,\,Y)\mapsto V \ot Y
\end{equation}
for all $X,\,Y\in \A$ and $V\in \A_\alpha$. Under isomorphism \eqref{main iso} this category $\A_\alpha$
corresponds to the induced autoequivalence $\Delta_\alpha$, i.e., 
\[
\Phi_{\A_\alpha} = \Delta_\alpha.
\]
\end{example}


For a finite group $G$ let $\Inn(G)\subset \Aut(G)$ denote the normal subgroup of inner 
automorphisms of $G$ and let $\Out(G) =\Aut(G)/\Inn(G)$.

\begin{proposition}
\label{kernel of induction}
The kernel of induction homomorphism  
\[
\Delta: \Aut(\Vec_G) \to \Aut^{br}(\Z(\Vec_G))
\] 
is $\Inn(G)$.
\end{proposition}
\begin{proof}
By Theorem~\ref{BrPic=Aut-br} and Example~\ref{induced bimodule}, 
the kernel of $\Delta$ consists of all autoequivalences $\alpha\in \Aut(\Vec_G)$ such that the 
$\Vec_G$-bimodule category $(\Vec_G)_\alpha$ is equivalent to the regular
$\Vec_G$-bimodule category $\Vec_G$. 

Let $\alpha = F_{(a,\zeta)}, a\in \Aut(G),\, \zeta\in H^2(G,\, k^\times)$ (
we use notation from Section~\ref{section: Fusion categories and their Brauer-Picard groups}). 
It follows from definition of $(\Vec_G)_\alpha$ (see \eqref{C alpha}) that any right $\Vec_G$-module equivalence
between $(\Vec_G)_\alpha$ and $\Vec_G$ is of the form $\delta_g \mapsto \delta_x \ot \delta_g,\, g\in G,$ for some
invertible $x\in G$.  This autoequivalence is compatible with the left
$\Vec_G$-module structure of $(\Vec_G)_\alpha$ if and only if $a$ is equal to the conjugation by $x$
and $\zeta$ is the trivial cohomology class. 
\end{proof}


Let $\C$ be a braided fusion category.  
Then  $\C$ is embedded into $\Z(\C)$ via $X\mapsto (X,\, c_{-,X})$, where $c$ denotes the braiding of $\C$.
In what follows we will identify $\C$ with a fusion subcategory of $\Z(\C)$ (the image of this embedding).
Left $\C$-module categories
can be viewed as $\C$-bimodule categories (analogously to how modules over a commutative
ring can be viewed as bimodules).  Invertible left $\C$-module categories form a subgroup
$\Pic(\C) \subset \BrPic(\C)$ called the {\em Picard} group of $\C$.  Note that the action of the group $\Aut^{br}(\C)$
on $\Pic(\C)$ factors through $\Out(\C)$.

\begin{remark}
\label{braided induction}
The restriction of the induction homomorphism \eqref{induction from Aut} to $\Aut^{br}(\C)$
is injective. 
\end{remark}

Let $\Aut^{br}(\Z(\C);\, \C) \subset \Aut^{br}(\Z(\C))$ be the subgroup consisting of braided 
autoequivalences of $\Z(\C)$ that restrict to the trivial autoequivalence of $\C$. 

The following result was established in \cite{DN}.
\begin{theorem}
\label{image of Pic}
The image of $\Pic(\C)$ under isomorphism \eqref{main iso} is  $\Aut^{br}(\Z(\C);\, \C)$.
\end{theorem}

The group $\Aut^{br}(\Z(\C))$ acts on the lattice of fusion subcategories of $\Z(\C)$. Let $\St(\C)$
denote the stabilizer of the subcategory $\C \subset \Z(\C)$ under this action.

\begin{proposition}
\label{stabilizer of C}
For a braided fusion category $\C$ we have
\begin{equation}
\label{Stab}
\St(\C) \cong \Pic(\C)\rtimes \Aut^{br}(\C).
\end{equation}
\end{proposition}
\begin{proof}
Observe that the subgroup $N:=\Aut^{br}(\Z(\C);\, \C)$ is normal in $\St(\C)$ and $N \cong \Pic(\C)$
by Theorem~\ref{image of Pic}.
Since the image of a braided autoequivalence of $\C$ under the induction homomorphism \eqref{induction from Aut}
belongs to $\St(\C)$ we  see from Remark~\ref{braided induction}  that $\St(\C)$ 
contains a subgroup $H \cong  \Aut^{br}(\C)$. 

Any $\alpha\in \St(\C)$ restricts to a braided autoequivalence $\tilde\alpha$ of $\C$. 
Let $\beta=\Delta_{\tilde\alpha}\in H$ be the element of $\Aut^{br}(\Z(\C))$ induced from $\tilde\alpha$.
The restriction of $\beta$ on $\C$ is $\tilde\alpha$, hence,  $\alpha \circ \beta^{-1} \in N$
restricts to a trivial autoequivalence of  $\C$.  This proves $\St(\C) = NH$, i.e., $\St(\C)$
is the semi-direct product of $N$ and $H$. 
\end{proof}

We can apply Proposition~\ref{stabilizer of C} to centers of Tannakian categories.
Let $G$ be a finite group.

\begin{corollary}
\label{found StRepG}
We have
\begin{equation}
\label{StRepG iso}
\St(\Rep(G)) \cong H^2(G,\,k^\times) \rtimes \Out(G)
\end{equation}
\end{corollary}
\begin{proof}
It was shown in \cite{Gr} that $\Pic(\Rep(G)) \cong H^2(G,\,k^\times)$. Combining this 
with  Proposition~\ref{Aut-br=Out} we get the result.
\end{proof}

Let us describe the above stabilizer $\St(\Rep(G))$ in terms convenient for computations.
Note that $\Z(\Rep(G)) \cong \Z(\Vec_G)$ and so we can consider the induction homomorphism
\[
\Delta :\Aut(\Vec_G) \to  \Aut^{br}(\Z(\Rep(G)). 
\]

\begin{lemma}
\label{parameterization of St}
$\St(\Rep(G)) = \Delta (\Aut(\Vec_G))$.
\end{lemma}
\begin{proof}
This follows from Propositions~\ref{Aut(VecG)}  and~\ref{kernel of induction}.
\end{proof}

For $a\in \Out(G)$ and $\zeta \in H^2(G,\, k^\times)$  let 
\begin{equation}
\label{dima}
\Delta_{(a,\,\zeta)}\in  \Aut^{br}(\Z(\Rep(G))) \cong \Aut^{br}(\Z(\Vec_G))
\end{equation}
denote the braided  autoequivalence induced from the tensor autoequivalence 
$F_{(a,\zeta)}$  of $\Vec_G$ introduced in Section~\ref{section: Fusion categories and their Brauer-Picard groups}. 
By Lemma~\ref{parameterization of St}, $\Delta_{(a,\,\zeta)}\in \St(\Rep(G))$. 

\begin{example}
\label{inductionexample}
We can compute the effect of  autoequivalence $\Delta_{a,\,\zeta}$ on objects of $\Z(\Vec_G)$.
Recall that objects of $\Z(\Vec_G)$ can be viewed as $G$-equivariant vector bundles of $G$, i.e., $G$-graded
vector spaces 
\[
V  =\bigoplus_{g\in G}\, V_g
\]
along with linear isomorphisms $\gamma(x,\,g): V_g \to V_{xgx^{-1}}, x,g\in G,$ satisfying  
\begin{equation}
\label{gamma}
\gamma(xy,\, g) = \gamma(x,\, ygy^{-1}) \circ \gamma(y,\,g),\qquad g,\,x,\,y\in  G.
\end{equation}
In particular, simple objects of $\Z(\Vec_G)$ are parameterized by pairs $(K,\,\chi)$, where $K$
is a conjugacy class of $G$ and $\chi$ is an irreducible character of the centralizer $C_G(g)$
of an element $g\in K$. Let $Z_{(K,\,\chi)}$ denote the corresponding simple object.

Let us denote
$\bigoplus_{g\epsilon G}\,\left( V_g,\,\{\gamma(x,\,g)\}_{x,g\in G}\right)$ a typical object in $\Z(\Vec_G)$.

We have 
\begin{equation}
\label{effect of induction}
\Delta_{(a,\zeta)}\left (\bigoplus_{g\in G}\,V_g,\{\gamma(x,\,g)\}_{x,g\in G} \right)
=\left( \bigoplus_{g\in G}V_{a^{-1}(g)},\{\widetilde{\gamma(x,\,g)}\}_{x,g\in G} \right),
 \end{equation}
where 
\begin{equation*}
\widetilde{\gamma(x,\,g)}=\gamma(a^{-1}(x),\,a^{-1}(g))\,\frac{\zeta(a^{-1}(x),\,a^{-1}(g))}{\zeta(a^{-1}(g),\,a^{-1}(x))}.
\end{equation*}
In particular, $\Delta_{(a,\zeta)}(Z_{(K,\,\chi)}) = Z_{(a(K),\,(\chi\circ a^{-1})\,\rho^g_a)}$, where $\rho^g_a,\, g\in a(K),$ is the linear
character of $C_G(g)$ given by 
\[
\rho^g_a(x) = \frac{\zeta(a^{-1}(x),\,a^{-1}(g))}{\zeta(a^{-1}(g),\,a^{-1}(x))}.
\]
\end{example}

\section{Action on the categorical Lagrangian Grassmannian}
\label{action on Grassmannian}

Let $\C$ be a non-degenerate braided fusion category.

\begin{definition}
A fusion subcategory $\D \subset \C$ is called {\em Lagrangian} if $\D$ is Tannakian and $\D =\D'$.
\end{definition}

It was shown in \cite{DGNO} that $\C$ contains a Lagrangian subcategory if and only if 
$\C$ is braided equivalent to the center of a pointed fusion category. 

Lagrangian subcategories of $\Z(\Vec_G)$ were classified in \cite{NN}.  
They are paramete\-ri\-zed by pairs $(N,\,\mu)$
where $N$ is a normal Abelian subgroup of $G$ and 
$\mu$ is a $G$-invariant cohomology class in $H^2(N,\,k^\times)$.
The Lagrangian subcategory $\mathcal{L}_{(N, \mu)}$ corresponding to the pair $(N,\, \mu)$
is identified with the subcategory of $G$-equivariant bundles $V=\oplus_{a\in N}\, V_a$ supported on $N$
whose $G$-equivariant structure \eqref{gamma} satisfies
\[
\gamma(x,\,a) = \frac{\mu(a,\,x)}{\mu(x,\,a)}\id_{V_a} 
\]
for all $a,\,x\in N$. 

\begin{example}
\label{canonical subcategory}
The canonical subcategory $\Rep(G) \subset \Z(\Vec_G)$ consisting of vector bundles supported
on the identity element of $G$ is $\mathcal{L}_{(1, 1)}$.
\end{example}


We have $\mathcal{L}_{(N, \mu)} \cong \Rep(G_{(N, \mu)})$ for some group $G_{(N, \mu)}$ such that $|G_{(N, \mu)}|=|G|$. 
The group $G_{(N, \mu)}$  is not isomorphic to $G$ in general. It can be  described as follows (see \cite{N} for details).
There exists a canonical homomorphism
\begin{equation}
\label{imageofmu}
H^2(N,\,k^{\times})^G \to H^2(G/N,\,\widehat{N}).
\end{equation}
Let $\nu \in H^2(G/N,\,\widehat{N})$
be the image of $\mu$ under this homomorphism. Then $G_{(N, \mu)}$ is an extension 
\[
1\to \widehat{N} \to G_{(N, \mu)} \to G/N \to 1
\]
corresponding to $\nu$.

\begin{remark}
\label{about GNB}
\begin{enumerate}
\item If $\mu \in H^2(N,\,k^{\times})^G$ is trivial then  $G_{(N, \mu)}$ is isomorphic to the semidirect product  $ \widehat{N} \rtimes G/N$.  
\item For non-degenerate $\mu$ the group $G_{(N, \mu)}$ first appeared in \cite{Da1}.
\end{enumerate}
\end{remark}

\begin{definition}
\label{Lagrangian Grassmannian}
Let $\C$ be a non-degenerate braided fusion category.
The set of Lagrangian subcategories of $\C$  will be called the {\em categorical Lagrangian Grassmannian} of $\C$.
\end{definition}

Let   $\mathbb{L}(G)$ denote the categorical Lagrangian Grassmannian of $\Z(\Vec_G)$.
Let 
\begin{equation}
\mathbb{L}_0(G):=\{ \L \in  \mathbb{L}(G) \mid \L \cong \Rep(G) \ \mbox{as a braided fusion category}\}.
\end{equation}
This set is non-empty since it contains the canonical subcategory $\Rep(G)\subset \Z(\Vec_G)$,
see Example~\ref{canonical subcategory}.

To simplify notation in what follows we will denote 
\begin{equation}
\label{Aa}
\Aa(G):= \Aut^{br}(\Z(\Vec_G)).  
\end{equation}
By Theorem~\ref{BrPic=Aut-br} we have 
\begin{equation}
\Aa(G)\cong \BrPic(\Vec_G) =\BrPic(\Rep(G)).
\end{equation}
Clearly, the group $\Aa(G)$ acts on the set $\mathbb{L}(G)$ and leaves the subset $\mathbb{L}_0(G)$ invariant. 

Let us denote $\C= \Z(\Vec_G)$. 

\begin{remark}
\label{trivial omega}
For every category $\mathcal{L} \in \mathbb{L}_0(G)$ the  algebra 
\begin{equation}
\label{AL}
A_\L := \Fun(G,\,k^\times)\in \mathcal{L}\cong \Rep(G)
\end{equation}
is commutative and separable (i.e., is an {\em \'etale} algebra in terminology of \cite{DMNO}) and
the fusion category $\C_{A_\L}$ of $A_\L$-modules in $\C$ is equivalent to $\Vec_G$. 
Indeed, by \cite{DGNO}, this category is pointed and has a faithful $G$-grading. Hence, $\C_{A_\L}$ is equivalent
to $\Vec_G^\omega$ for some $\omega\in Z^3(G,\, k^\times)$.  By \cite{DMNO} $\Z(\C_{A_\L}) \cong  \Z(\Vec_G)$
and, hence, $\C_{A_\L}$ and $\Vec_G$ are categorically Morita equivalent \cite{ENO2}.  This implies
that  $\omega$ is cohomologically trivial.
\end{remark} 

\begin{proposition}
\label{transitive action}
The action of $\Aa(G)$ on $\mathbb{L}_0(G)$ is transitive.
\end{proposition}
\begin{proof}
Let $\mathcal{L}_1,\,\mathcal{L}_2\in \mathbb{L}_0(G)$  be Lagrangian subcategories of  $\C$ 
and let $A_1$ and $A_2$ be the corresponding \'etale algebras in $\C$ defined in \eqref{AL}. 
By Remark~\ref{trivial omega}
$\C$-module categories 
$\C_{A_1}$ and $\C_{A_2}$ are equivalent to $\Vec_G$.  Pick a tensor  equivalence 
\[
\phi: 
\C_{A_1} \xrightarrow{\sim} \C_{A_2}.  
\]
It follows from the results of \cite{DMNO} that
there are braided equivalences 
\[
\Phi_i : \C \xrightarrow{\sim} \Z(\C_{A_i}),\, i=1,2. 
\]
Let $ \alpha:= \Phi_2^{-1} \circ  \Delta_\phi \circ \Phi_1 \in \Aut^{br}(\C)$, where $\Delta_\phi: 
\Z(\C_{A_1}) \to \Z(\C_{A_2})$
is the braided equivalence induced from $\phi$. 
Then $\alpha(A_1) \cong  A_2$. 
Note that $A_i$
is isomorphic, as an object of $\C$, to the regular object in $\L_i,\, i=1,2$. Hence, $\alpha(\L_1) =\L_2$, 
which proves the statement. 
\end{proof}
 
Thus, the image of $\Aa(G)$ is a transitive subgroup of $\Sym(\mathbb{L}_0(G))$. 
Let 
\begin{equation} 
\Aa_0(G):= \St(\Rep(G))
\end{equation}
 denote the stablizer of  the canonical
Lagrangian subcategory $\Rep(G) \subset \C$ in $\Aa(G)$. 
By Corollary~\ref{found StRepG},
\begin{equation}
\label{A)G}
\Aa_0(G) \cong H^2(G,\,k^\times) \rtimes \Out(G).
\end{equation}
Since the cardinality of a transitive set is equal to the index of the stabilizer of a point, 
we have
\begin{equation}
\label{index of A0G}
[\Aa(G): \Aa_0(G)] = |\mathbb{L}_0(G)|.
\end{equation}

The next Corollary allows to find the order of the Brauer-Picard group.

\begin{corollary}
\label{orderofBrPic}
Let $G$ be a finite group.  Then
\[
|\Aa(G)| = |H^2(G,\,k^\times)| \cdot |\Out(G)| \cdot |\mathbb{L}_0(G)|.
\]
\end{corollary}

\begin{corollary}
\label{L(G) is trivial}
Let $G$  be a finite group without normal Abelian subgroups (e.g., a simple non-Abelian group). Then
\[
\Aa(G) \cong H^2(G,\,k^\times) \rtimes \Out(G). 
\]
\end{corollary}

\begin{corollary}
Let $G$ be a finite non-Abelian simple group. Then the  Brauer-Picard group of $\Vec_G$ is solvable.
\end{corollary}
\begin{proof}
This is a consequence of the Schreier conjecture \cite[p.133]{DM} stating that $\Out(G)$ is solvable 
(this conjecture is verified using the classification of finite simple groups). 
\end{proof}

\begin{remark}
It can happen that the group $\Aa(G)$ is trivial. This is the case  for every simple group $G$
such that both $\Out(G)$ and $H^2(G,\,k^\times)$ are trivial.  Among the groups that have these properties
are the Mathieu group $M_{11}$ and the Fischer-Griess Monster group. 

On the other hand, if $G$ is a $p$-group of order $> 2$ then $\Aa(G)$ is non-trivial, since 
in this case the group $\Out(G)$ is non-trivial \cite{Ga}.
\end{remark}

The next Proposition describes the action of $\Aa_0(G)$ on $\mathbb{L}(G)$.

By Lemma~\ref{parameterization of St}, 
elements of $\Aa_0(G)$ are precisely braided autoequivalences induced from $\Aut(\Vec_G)$  
and so are of the form $\Delta_{(a,\,\zeta)}$ for some 
$a\in \Out(G)$ and $\zeta \in H^2(G,\,k^\times)$, see \eqref{dima}.

\begin{proposition}
\label{action of A0G on L}
Let $\mathcal{L}_{(N,\,\mu)}$ be a Lagrangian subcategory of $\Z(\Vec_G)$. Then
\[
\Delta_{(a,\,\zeta)} (\L_{(N,\,\mu)}) = \L_{(a(N),\,\mu^a \zeta^{a})}.
\]
\end{proposition}
\begin{proof}
Let us apply $\Delta_{(a,\,\zeta)}$ to an object 
$\left(\bigoplus_{g\in N}V_g,\{\gamma(x,\,g)\}\right)$ of $\mathcal{L}_{(N,\,\mu)}$. 
Using Example~\ref{inductionexample}, we obtain
\[
\Delta_{(a,\zeta)}\left (\bigoplus_{g\in G}\,V_g,\{\gamma(x,\,g)\}_{x,g\in G} \right)
=\left( \bigoplus_{g\in G}V_{a(g)},\{\widetilde{\gamma(x,\,g)}\}_{x,g\in G} \right),
\]
where 
 \begin{equation}
 \widetilde{\gamma(x,\,g)}=\gamma(a^{-1}(x),\,a^{-1}(g)) \,
 \frac{\zeta^a(x,\,g)} {\zeta^a(g,\,x)} =  \frac{\mu^a(x,\,g)} {\mu^a(g,\,x)} \, \frac{\zeta^a(x,\,g)} {\zeta^a(g,\,x)}
\id_{V_g},
 \end{equation}
 for all $x,\,g\in N$, which implies the result.
\end{proof}

\begin{proposition}
\label{M(L,mu) on L}
Let $L$ be a subgroup of $G\times G^\op$ and let $\mu$ be a $2$-cocycle in $Z^2(L,\,k^\times)$
satisfying conditions of Proposition~\ref{M(L,mu) inv}.  Let $\M(L,\,\mu)$ denote the corresponding element
of $\BrPic(\Vec_G)$ and let $\alpha_{(L,\mu)}$ be the braided autoequivalence of $\Z(\Vec_G)$ corresponding   to  
$\M(L,\,\mu)$ upon the isomorphism \eqref{main iso}. Then
\begin{equation}
\label{alpha 11}
\alpha_{(L,\mu)}(\L_{(1,\,1)}) = \L_{(L_1,\, \mu|_{L_1\times L_1})},
\end{equation}
where $L_1 = L \cap(G \times 1)$.
\end{proposition}
\begin{proof}
Let $\alpha \in \Aut^{br}(\Z(\Vec_G))$ be an autoequivalence such that $\alpha(\L_{(1,\,1)}) = \L_{(N,\nu)}$. 
Let $\M_\alpha \in \BrPic(\Vec_G)$ be the corresponding invertible $\Vec_G$-bimodule category
(see discussion after Theorem~\ref{BrPic=Aut-br}). By Remark~\ref{discussing inverse} 
$\M_\alpha$ is identified, as a left $\Vec_G$-module category,  with the category of modules over the twisted subgroup algebra 
$(kN)_\nu$, i.e., $\M_\alpha \cong \M(N,\, \nu)$. Comparing this with Remark~\ref{restriction to one-sided}
we obtain the result.
\end{proof}

\begin{remark}
Here is an alternative way to deduce Proposition~\ref{M(L,mu) on L}. 
It was observed in \cite{NN} that there is a bijection between the set of $\Vec_G$-module categories $\M$
such that the dual category $(\Vec_G)^*_\M$ is pointed and $\mathbb{L}(G)$.  This bijection is equivariant
with respect to the group isomorphism $\BrPic(\Vec_G) \xrightarrow{\sim} \Aut^{br}(\Z(\Vec_G))$. 
This implies \eqref{alpha 11}. 
\end{remark}

\section{Examples: symmetric and alternating groups}
\label{examples: symmetric}

\subsection{Symmetric group $S_3$}

It is well known that $H^2(S_3,\,k^{\times}) =0$, see \cite[Theorem 2.12.3]{K} and $\Out(S_3)=1$, hence $\Aa_0(S_3)=1$.
The only nontrivial normal Abelian subgroup of $S_3$ is isomorphic to $\GrpCthree$.
It follows that $|\mathbb{L}_0(S_3)|=2$.
By Corollary~\ref{orderofBrPic} $|\Aa(S_3))|=2$, thus 
\begin{equation}
\label{BrPic S3}
\Aa(S_3) \cong \GrpCtwo.
\end{equation}

\begin{example}
\label{S3 example}
By \eqref{BrPic S3} we have $\BrPic(\Rep(S_3)) \cong \mathbb{Z}/2\mathbb{Z}$.
There exists a non-trivial $\mathbb{Z}/2\mathbb{Z}$-extension of the fusion category $\Rep(S_3)$, 
namely the category  $\C(sl(2),\, 4)$ of highest weight integrable modules over the affine 
Lie algebra $\widehat{sl}(2)$ of level $4$. This is a weakly integral fusion category of dimension $12$.
Its  simple objects lying in the non-trivial component have dimension $\sqrt{3}$.
\end{example}

\subsection{Symmetric group $S_4$}
\label{section S4}

It is known that $H^2(S_4,\,k^{\times})\cong \GrpCtwo$, see \cite[Theorem~2.12.3]{K}, and $\Out(S_4)=1$.
Hence, $\Aa_0(S_4)\cong \GrpCtwo$.

The set $\mathbb{L}(S_4)$ consists of three subcategories (see \cite[Example~5.2]{NN}):
\[
\L_{(1,1)},\, \L_{(\GrpV, 1)},\,\mbox{and } \L_{(\GrpV, \mu)}, 
\]
where $\GrpV$ is identified with a normal subgroup of $S_4$ and 
$\mu$ is the non-trivial class in $H^2(\GrpV,\,k^{\times})$. 

We claim that $\Aa(S_4)$ permutes the two later categories.  To prove this 
we apply Proposition~\ref{action of A0G on L}.
It is enough to check that  the  restriction 
\[
H^2(S_4,\,k^{\times})\to H^2(\GrpV,\,k^{\times})
\] 
is injective. 
This follows from Theorem~\ref{5termexactseq} since
$H^1(S_3,\, \GrpV)=1$.

Thus, the map $\Aa(S_4) \to \Sym(\mathbb{L}_0(S_4))$ is injective and, hence   
\begin{equation}
\label{BrPic S4}
\Aa(S_4) \cong S_3.
\end{equation}

\subsection{Alternating group $A_4$}

It is known that $H^2(A_4,\,k^{\times})\cong \GrpCtwo$, see \cite{K}, and $\Out(A_4)\cong \GrpCtwo$.
The set $\mathbb{L}_0(A_4)$ consists of three Lagrangian subcategories (see  \cite[Example~5.2]{NN}):
\[
\L_{(1,1)},\, \L_{(\GrpV, 1)},\,\mbox{and } \L_{(\GrpV, \mu)}, 
\]
where $\GrpV$ is identified with the Sylow $2$-subgroup  of $A_4$ and 
$\mu$ is the non-trivial class in $H^2(\GrpV,\,k^{\times})$. 
We claim that $H^2(A_4,\,k^{\times}) \subset \Aa_0(A_4)$ permutes the two later categories.  To prove this 
we apply Proposition~\ref{action of A0G on L}.
It is enough to check that  the  restriction 
\[
H^2(A_4,\,k^{\times})\to H^2(\GrpV,\,k^{\times})
\] 
is injective, which  follows immediately from Theorem~\ref{Sylow restriction}.
Hence, the map 
\[
\pi:\Aa(A_4)\rightarrow \Sym(\mathbb{L}_0(A_4))\cong S_3
\]
is surjective.

By Corollary~\ref{orderofBrPic}  $\Aa(A_4)$ is a non-Abelian group of order $12$.
Its Sylow $2$-subgroup is isomorphic to $\GrpV \cong \Aa_0(A_4)$.  Furthermore, $\Aa(A_4)$
contains a normal subgroup of order $2$ (the kernel of $\pi$).  It is easy to check 
that the only group of order $12$ with above properties is the dihedral group of order $12$. 
Thus,
\begin{equation}
\label{BrPic A4}
\Aa(A_4)\cong D_{12}.
\end{equation}

\begin{remark}
Let $G=A_n$ or $S_n$, where  $n \geq 5$. Then $G$ has no normal Abelian subgroups and $\Aa(G) =H^2(G,\, k^\times) \rtimes \Out(G)$
by  Corollary~\ref{L(G) is trivial}.  The groups $H^2(G,\, k^\times)$ and $\Out(G)$ in this case are well known, see 
\cite[Section 2.12]{K}:

\begin{equation*}
 H^2(S_n,k^{\times})=  
  \mathbb{Z}/2\mathbb{Z}, \qquad \Out(S_n)=
 \begin{cases}
  1 & n\neq 6\\
  \mathbb{Z}/2\mathbb{Z} & n=6,
 \end{cases}
\end{equation*}

\begin{equation*}
 H^2(A_n,k^{\times})=
 \begin{cases}
  \mathbb{Z}/2\mathbb{Z} & n \neq 6,7\\
  \mathbb{Z}/6\mathbb{Z} &  n =6,7, 
 \end{cases}
\qquad
 \Out(A_n)=
 \begin{cases}
   \mathbb{Z}/2\mathbb{Z} & n\neq 6\\
  \mathbb{Z}/2\mathbb{Z}\times  \mathbb{Z}/2\mathbb{Z} & n=6.
 \end{cases}
\end{equation*}
\end{remark}

\section{Examples: non-abelian groups of order $8$}
\label{examples: order 8}

\subsection{Dihedral group $D_8$}

The following is a standard representation of   $D_8$ by generators and relations:

\begin{equation*}
D_8=  \langle r,s \mid r^4=s^2=1, sr=r^{-1}s \rangle.
\end{equation*}
We have  $\Out(D_8)\cong\GrpCtwo$, where the nontrivial element
is represented by the  automorphism $a$  given by
\[
a(r) =r,\qquad a(s) = sr.
\]
Also, $H^2(D_8,\,k^{\times})\cong\GrpCtwo$, see \cite[Theorem~2.11.4]{K}.
Thus, $\Aa_0(D_8)\cong\GrpV$.

By analyzing subgroup structure of $D_8$ we see that $\Z(\Vec_{D_8})$ has
seven Lagrangian subcategories,
\begin{equation}\label{D8Llist}
 \L_{(1,1)},\, \L_{(\langle r\rangle, 1)},\, \L_{(\langle s,r^2\rangle, 1)},\, \L_{(\langle sr,r^2\rangle, 1)},\, \L_{(\langle s,r^2\rangle, \mu_1)},\,\\
 \L_{(\langle sr,r^2\rangle, \mu_2)},\,\mbox{and } \L_{(\langle r^2\rangle, 1)},
\end{equation}
where $\mu_1, \mu_2$ denote nontrivial cohomology classes of the respective subgroups.

The following fact was established in \cite[Example~5.1]{NN}.  It is included here for the reader's convenience.

\begin{lemma}
 The Lagrangian subcategories in $\mathbb{L}_0(D_8)$ are precisely the following:
 \begin{equation}
 \label{D8L0list}
 \L_{(1,1)},\, \L_{(\langle r\rangle, 1)},\, \L_{(\langle s,r^2\rangle, 1)},\, \L_{(\langle sr,r^2\rangle, 1)},\, \L_{(\langle s,r^2\rangle, \mu_1)},\,\mbox{and }\,
  \L_{(\langle  sr,r^2\rangle, \mu_2)}.
\end{equation}
\end{lemma}
\begin{proof}
Clearly, $\L_{(1,1)} \in \mathbb{L}_0(D_8)$. Using Remark~\ref{about GNB} we see that
Lagrangian subcategories  $\L_{(\langle r\rangle, 1)},\, \L_{(\langle s,r^2\rangle, 1)}$, and $\L_{(\langle sr,r^2\rangle, 1)}$
are all equivalent to $\Rep(D_8)$., i.e., belong to $\mathbb{L}_0(D_8)$. 
To see that subcategories $\L_{(\langle s,r^2\rangle, \mu_1)}$ and  $\L_{(\langle  sr,r^2\rangle, \mu_2)}$ are in $\mathbb{L}_0(D_8)$
note that each of the is equivalent to a category $\Rep(G)$, where $G$ is a non-Abelian group of order $8$ having a normal subgroup
isomorphic to $ \GrpCtwo\times \GrpCtwo$. The only group $G$ with this property is $D_8$.

Finally,  $\L_{(\langle r^2\rangle, 1)}$ is equivalent to $\Rep(\GrpCtwo\times\GrpCtwo\times\GrpCtwo)$  and so is not in $\mathbb{L}_0(D_8)$.
\end{proof}


\begin{lemma}
\label{D8isomorphism}
The restriction map $H^2(D_8,\,k^{\times})\to H^2(\mathbb{Z}/2\mathbb{Z} \times \mathbb{Z}/2\mathbb{Z},\,k^{\times})$ is an isomorphism.
 \end{lemma}
\begin{proof}
By Theorem~\ref{Sylow restriction} the restriction map $M(S_4)\rightarrow M(D_8)$ is injective.
We saw in Section~\ref{section S4} that the restriction
\[
H^2(S_4,\,k^{\times})\rightarrow H^2(\mathbb{Z}/2\mathbb{Z}\times \mathbb{Z}/2\mathbb{Z},\,k^{\times})
\]
 is an isomorphism.
This implies the claim.
\end{proof}

Let us describe the action of  $\Aa_0(D_8)$ on $\Sym(\mathbb{L}_0(D_8))$.

Let $\mu \in H^2(D_8,\,k^{\times})\subset \Aa_0(D_8)$ and $a\in \Out(D_8) \subset \Aa_0(D_8)$
be the generators of $\Aa_0(D_8) \cong H^2(D_8,\,k^{\times})\rtimes \Out(D_8) \cong \GrpCtwo\times\GrpCtwo$.
By  Lemma~\ref{D8isomorphism} $\mu$ maps $\L_{(\langle s,r^2\rangle, 1)}$ to $\L_{(\langle s,r^2\rangle, \mu_1)}$
and  $\L_{(\langle sr,r^2\rangle, 1)}$ to $\L_{(\langle sr,r^2\rangle, \mu_2)}$.
Also  $a$ maps $\L_{(\langle s,r^2\rangle, 1)}$ to $\L_{(\langle sr,r^2\rangle, 1)}$ and  $\L_{(\langle s,r^2\rangle, \mu_1)}$ to $\L_{(\langle sr,r^2\rangle, \mu_2)}$.

Thus, $\pi:\Aa(D_8)\rightarrow \Sym(\mathbb{L}_0(D_8)\cong S_6$ is injective, i.e., $\Aa(D_8)$ is a transitive
subgroup of $S_6$.   By Corollary~\ref{orderofBrPic} $|\Aa(D_8)|=24$.

Enumerating  Lagrangian subcategories in the list \eqref{D8L0list} we have
\[
\Aa_0(D_8) = \{1,(35)(46),(34)(56),(36)(45)\}
\]
as a subgroup of $S_6$.  Other stabilizers of points in $\mathbb{L}_0(D_8)$
are the following conjugates of $\Aa(D_8)$:
\begin{equation}
\label{3 subgroups}
\{1,(12)(56),(15)(26),(16)(25)\} \quad \mbox{and} \quad \{1,(12)(34),(13)(24),(14)(23)\}.
\end{equation}
Note that the  elements
\[
s_1 := (13)(24),\quad s_2:=(15)(26),\quad s_3:=(14)(23)
\]
satisfy the usual symmetric group relations
\[
(s_1s_2)^3 =1,\, (s_2s_3)^3=1,\, s_1s_3=s_3s_1,\, s_1^2= s_2^2=s_3^2=1,
\]
and, hence, generate a subgroup isomorphic to $S_4$. Thus,
\begin{equation}
\label{BrPic D8}
\Aa(D_8)\cong S_4.
\end{equation}

\begin{example}
Using Remark~\ref{rel prime ext} we can construct a non-trivial $\mathbb{Z}/3\mathbb{Z}$-extension
of $\Rep(D_8)$ (or $\Vec_{D_8}$). Any such an extension is an integral fusion category
of dimension $24$ all whose non-invertible simple objects  have dimension $2$.
\end{example}

\subsection{Quaternion group $Q_8$}
\label{Q8calculations}
 
Let $Q_8=\{\pm 1,\pm i,\pm j,\pm k\}$ denote the quaternion group.

It is known that $H^2(Q_8,\,k^{\times})=1$, see \cite{K}, and $\Out(Q_8)=S_3$. Thus, $\Aa_0(D_8)\cong S_3$. 

It is easy to find  Lagrangian subcategories of $\Z(\Vec_{Q_8})$.
There are five normal Abelian subgroups of $Q_8$:
\[
1,\, \langle -1 \rangle,\, \langle i \rangle,\,  \langle j \rangle, \quad\mbox{and}\quad  \langle k \rangle.
\]
Since $H^2(N,\,k^{\times})=1$ for each of these subgroups there are precisely five Lagrangian subcategories of $\Z(\Vec_{Q_8})$.
By Remark~\ref{about GNB} the Lagrangian subcategory $\L_{(N,1)}$ is equivalent to $\Rep(\widehat{N}\rtimes (Q_8/N))$.
But $Q_8$ is not isomorphic to  any non-trivial semidirect product.
Thus, $\Rep(\widehat{N}\rtimes (Q_8/ N))$ is equivalent to $\Rep(Q_8)$ if and only if $N$ is trivial.
It follows that there is precisely one Lagrangian subcategory in $\mathbb{L}_0(Q_8)$. Thus,
\begin{equation}
\label{BrPic Q8}
\Aa(Q_8) \cong S_3.
\end{equation}

\begin{remark}
Since $\BrPic(\Rep(Q_8)) \cong \BrPic(\Vec_{Q_8}) =\Out(Q_8)$ we see that categories $\Rep(Q_8)$ and $\Vec_{Q_8}$ have
no non-trivial extensions. 
\end{remark}

\section{Examples: groups of order $pq$}
\label{order pq section}

Let $p,\,q$ be prime numbers such that $q \equiv 1 (\mod p)$.  It is well known that there is
a unique (up to an isomorphism) finite group $G$ of order $pq$, namely $G = \mathbb{Z}/q\mathbb{Z}
\rtimes \mathbb{Z}/p\mathbb{Z}$.

We will need the following  presentation of $G$ by generators and relations:
 \begin{equation}
 \label{pq proporties}
  G= \langle x,y \mid x^q=y^p=1 \text{ and }  yxy^{-1}=x^a \rangle,
 \end{equation}
 for a fixed  $a$ such that $a^p\equiv (1 \mod q)$.

\begin{lemma}
The group $\Out(G)$ is isomorphic to $\mathbb{Z}/\frac{q-1}{p}\mathbb{Z}$.
\end{lemma}
\begin{proof}
It is clear that any automorphism $\alpha$ of $G$ maps $\langle x \rangle$ to itself.  Since all subgroups of $G$ of order $p$
are conjugate to each other it follows that the composition of $\alpha$ with some inner automorphism of $G$ maps 
$\langle y \rangle$ to itself. Thus, modulo an inner automorphism, $\alpha$ is given by
 \begin{equation}
 \label{our alpha}
\alpha(x)=x^m,\quad \alpha(y)=y^n
\end{equation}
for some $m,\,n$ such that $1 \leq m < q$ and $1\leq n < p$. It is straightforward to check that in order to preserve
defining relations \eqref{pq proporties} of $G$ we must have $n=1$.  Also, automorphisms \eqref{our alpha}
with $m=a^i, i=0,\dots, p-1$ are inner. Thus,
\[
\Out(G) \cong (\mathbb{Z}/q\mathbb{Z})^\times /(\mathbb{Z}/p\mathbb{Z}) \cong 
\mathbb{Z}/\tfrac{q-1}{p}\mathbb{Z},
\]
as required.
\end{proof}

It is known that $H^2(G,\,k^{\times})=1$, see \cite[Corollary~2.1.3]{K}.
The only normal Abelian subgroups of $G$ are 1 and $\mathbb{Z}/q\mathbb{Z}$.
Hence, $\Z(\Vec_G)$ contains  precisely two Lagrangian subcategories.
\[
\L_{(1,1)}\quad  \mbox{and }\quad  \L_{(\mathbb{Z}/q\mathbb{Z}, 1)}. 
\] 
This implies that $|\Aa(G)|=\tfrac{2(q-1)}{p}$.

\begin{lemma}
\label{order 2 elements}
Any $\alpha\in \Aa(G)$ such that $\alpha\not\in \Aa_0(G)$   has order $2$.
\end{lemma}
\begin{proof}
The condition $\alpha\in \not\in \Aa_0(G)$  means that $\alpha$ permutes the pair of Lagrangian
subcategories of $\Z(\Vec_G)$.

Consider the subgroup $L \subset G \times G^\op$ generated by $(y,\,y^{-1})$ and $\langle x\rangle \times
\langle x\rangle$. We have
\[
L \cong (\mathbb{Z}/q\mathbb{Z} \times \mathbb{Z}/q\mathbb{Z}) \rtimes \mathbb{Z}/p\mathbb{Z}.
\]
By Theorem~\ref{5termexactseq} the restriction map 
\[
H^2(L,\,k^{\times}) \to H^2(\mathbb{Z}/q\mathbb{Z} \times \mathbb{Z}/q\mathbb{Z},\,k^{\times})
\cong \mathbb{Z}/q\mathbb{Z}
\]
is an isomorphism.  By Corollary~\ref{order 2}  the braided auto-equivalence $\alpha_{(L,\mu)}$ 
corresponding to the $\Vec_G$-bimodule
category $\M(L,\,\mu)$ has order $2$ in $\Aa(G)$ for any non-trivial $\mu \in H^2(L,\,k^{\times})$.

It is straightforward to check that  every $\alpha\not\in \Aa_0(G)$  is isomorphic to some  $\alpha_{(L,\mu)}$ 
(there are  $\frac{q-1}{p}$ isomorphism classes of such equivalences). 
\end{proof}

Lemma~\ref{order 2 elements} implies that
\begin{equation}
\label{BrPicpq}
\Aa(G)\cong D_{\frac{2(q-1)}{p}}.
\end{equation}

\begin{example}
Suppose that $p,\,q$ are odd.   Using Remark~\ref{rel prime ext} we conclude that any ``reflection" 
element of  $\BrPic(\Vec_G)\cong D_{\frac{2(q-1)}{p}}$ gives rise to a non-trivial $\mathbb{Z}/2\mathbb{Z}$-extension
of $\Vec_G$.   Any such an extension is a weakly integral fusion category of dimension $2pq$.
The non-trivial component of this extension contains $p$ classes of simple objects of dimension $\sqrt{q}$.
\end{example}


\section{Examples : dihedral groups $D_{2n}$ where $n$ is odd}
\label{examples: dihedral}

Recall that for any integer $n\geq 3$ we denote by $D_{2n}$ the dihedral group on $n$ vertices.
That is,
\begin{equation}
\label{properties of dihedral group}
D_{2n}= \langle r,\,s \mid  r^n=1,\, s^2=1,\, \text{ and } (sr)^2=1\rangle.
\end{equation}


Let $k$ be the number of distinct prime divisors of $n$.

\begin{lemma}
$\Out(D_{2n})\cong (\mathbb{Z}/n\mathbb{Z})^\times /\{ \pm 1\}$. 
\end{lemma}
\begin{proof}
For any $i$ relatively prime to $n$ consider $a_i\in \Aut(D_{2n})$ given by
\[
a_i(s)=s,\, a_i(r)=r^i.
\]
It is straightforward to check that the automorphisms  $a_i$ and $a_j$ are congruent modulo an
inner automorphism of $D_{2n}$ if and only if $i\equiv  -j (\mod n)$ and that every outer automorphism
of $D_{2n}$ is congruent to some $a_i$.  This implies the result.
\end{proof}

When $n$ is odd, $H^2(D_{2n},k^{\times})=0$ \cite[Proposition~2.11.4]{K}.

\begin{corollary}
$\St(\Rep(D_{2n})) \cong  (\mathbb{Z}/n\mathbb{Z})^\times /\{ \pm 1\}$.
\end{corollary}

\begin{proposition}
\label{L0 in dihedral case}
The set $\mathbb{L}_0(D_{2n})$ consists of subcategories $\L_{(\langle r^b\rangle,\,1)}$ 
where $b$ divides $n$ and $\frac{n}{b}$ and $b$ are relatively prime.
\end{proposition}
\begin{proof}  
Lagrangian subcategories of $\Z(\Vec_{D_{2n}})$ are all of the form
$\mathcal{L}_{(\langle r^b \rangle ,1)}$, where $b$ divides $n$.
By Remark~\ref{about GNB}  the category  $\mathcal{L}_{(\langle r \rangle,1)}$ is equivalent to
\begin{displaymath}
\Rep((\mathbb{Z}/\frac{n}{b}\mathbb{Z}\times \mathbb{Z}/b\mathbb{Z})\rtimes \GrpCtwo).
\end{displaymath}
The latter category is equivalent to $\Rep(D_{2n})$ if and only if the group $\mathbb{Z}/\frac{n}{b}\mathbb{Z}\times \mathbb{Z}/b\mathbb{Z}$ has an element of order $n$,  which is the case precisely when $\frac{n}{b}$ and $b$ are relatively prime.
\end{proof}

\begin{remark}
Note that $\Aa_0(D_{2n})$ stabilizes {\em all} Lagrangian subcategories in  $\mathbb{L}_0(D_{2n})$
and, hence, it is a normal subgroup of $\Aa(D_{2n})$. 
\end{remark}

\begin{lemma}
\label{image is elementary abelian}
The image of $\Aa(D_{2n})$ in $\Sym(\mathbb{L}_0(D_{2n}))$ is isomorphic to $(\mathbb{Z}/2\mathbb{Z})^k$.
\end{lemma} 
\begin{proof}
Let $\alpha \in \Aa(D_{2n})$  be a braided autoequivalence  such that $\alpha \not\in \St(\Rep(D_{2n}))$.
It suffices to show that the image of $\alpha$ in $\Sym(\mathbb{L}_0(D_{2n}))$ has order $2$. 
For this end, observe that for any fixed $\L_0 \in \mathbb{L}_0(D_{2n})$  all the dimensions
$\dim(\L_0\cap \L),\, \L \in \mathbb{L}_0(D_{2n})$ are distinct (they correspond to subsets of the 
set  of prime divisors of $n$).  Since $\alpha$ preserves dimensions of subcategories,
we have 
\[
\dim(\L \cap \alpha(\L)) = \dim (\alpha(\L) \cap \alpha^2(\L)),
\]
and, hence, $\alpha^2(\L) =\L$ for any $\L \in  \mathbb{L}_0(D_{2n})$.
\end{proof}

\begin{corollary}
There is a short exact sequence
\begin{equation}
\label{D2n sequence}
1 \to (\mathbb{Z}/n\mathbb{Z})^\times /\{ \pm 1\} \to \Aa(D_{2n}) \to  (\mathbb{Z}/2\mathbb{Z})^k \to 1. 
\end{equation} 
\end{corollary}

\begin{remark}
When $k=1$, i.e., when $n$ is a prime power, an argument similar to that in Section~\ref{order pq section}
shows that $\Aa(D_{2n}) \cong (\mathbb{Z}/n\mathbb{Z})^\times /\{ \pm 1\} \rtimes \mathbb{Z}/2\mathbb{Z}$,
i.e.,  $\Aa(D_{2n})$ is a generalized dihedral group. We conjecture that in general the sequence  \eqref{D2n sequence}
splits, i.e., 
\[
\Aa(D_{2n}) \cong (\mathbb{Z}/n\mathbb{Z})^\times /\{ \pm 1\} \rtimes (\mathbb{Z}/2\mathbb{Z})^k.
\]
\end{remark}

\bibliographystyle{ams-alpha}

\end{document}